\documentclass[11pt]{article}%
\usepackage{amsmath,amssymb,amsthm,amsfonts}
\usepackage{wasysym}
\usepackage{epsfig}
\usepackage{psfrag}
\usepackage{subfigure}
\usepackage{epstopdf}
\usepackage{amsmath}
\usepackage{amsfonts}
\usepackage{amssymb}
\usepackage{graphicx}%

\newtheorem{thm}{Theorem}
\theoremstyle{definition}
\newtheorem{definition}{Definition}

\begin{document}

\title{Peixoto's Structural Stability Theorem:\\ The One-dimensional Version}
\author{Aminur Rahman\\Department of Mathematical Sciences\\New Jersey Institute of Technology\\Newark, NJ 07102-1982\\ar276@njit.edu }
\date{}
\maketitle

\begin{abstract}
Peixoto's structural stability and density theorems represent milestones in
the modern theory of dynamical systems and their applications. Despite the
importance of these theorems, they are often treated rather superficially, if
at all, in upper level undergraduate courses on
dynamical systems or differential equations. This is mainly because of the
depth and length of the proofs. In this note/module, we formulate and prove
the one-dimensional analogs of Peixoto's theorems in an intuitive and fairly
simple way using only concepts and results that for the most part should be
familiar to upper level undergraduate students in the mathematical sciences or
related fields. The intention is to provide students who may be interested in
further study in dynamical systems with an accessible one-dimensional
treatment of structural stability theory that should help make Peixoto's
theorems and their more recent generalizations easier to appreciate and understand.

\end{abstract}

\section{Introduction}

Inspired by the pioneering work of \cite{Andronov-Pontryagin37} and encouraged by
Solomon Lefschetz, the Brazilian engineer and mathematician, Maur\'{\i}cio
Matos Peixoto, formulated and proved the first global characterization of
structural stability and its density properties on smooth surfaces
\cite{Peixoto62} in terms that have become synonymous with the modern theory
of dynamical systems and blazed a path for myriad extensions and
generalizations. One of the most powerful aspects of Peixoto's theorems is the
way it uses local properties to characterize global features of dynamical
systems. Both the structural stability and density theorems can be combined as
follows (see Perko, etc.):

\pagebreak

\noindent\textbf{Theorem P} ({Peixoto's Structural Stability and Density
Theorems)}. \textit{Let }$\dot{x}=f(x)$\textit{\ be a }$C^{1}$%
\textit{\ dynamical system on a smooth closed surface }$M$\textit{. Then the
dynamical system is }$C^{1}$\textit{-structurally stable if and only if it
satisfies the following properties}:

\begin{itemize}
\item[(i)] \textit{All recurrent behavior is confined to finitely many fixed
points and periodic orbits, all of which are hyperbolic}.

\item[(ii)] \textit{There are no separatrices, i.e. orbits connecting saddle
points}.
\end{itemize}

\noindent\textit{Moreover, if }$M$\textit{\ is orientable, then the set of
structurally stable systems is }$C^{1}$\textit{\ open and dense in the
collection of all dynamical systems on the surface}.

\medskip

This theorem involves mathematical concepts unfamiliar to many advanced
undergraduates interested in studying dynamical systems, and the proof is
quite long and complicated. Given the importance of the results, both from a
theoretical and applied perspective, the much simpler one-dimensional analog
treated in what follows is likely to prove useful for understanding Peixoto's
theorems and their generalizations, which comprise an essential part of the
modern theory of dynamical systems and its applications.

Consider a dynamical system on the circle $\mathbb{S}^{1}$, to which all
smooth closed one-dimensional manifolds (curves) are equivalent. We can
represent this as the unit interval on the real line with the end points
identified:
\begin{equation}
\mathbb{S}^{1}=\mathbb{R}/\mathbb{Z}%
\end{equation}
where $\mathbb{R}/\mathbb{Z}$ denotes the real numbers mod$1$; i.e. for
$x,y\in\mathbb{R}$, $x\sim y\Leftrightarrow x\equiv y(\text{mod}%
1)\Leftrightarrow x-y\in\mathbb{Z}$.

Now let $f:\mathbb{R}\rightarrow\mathbb{R}$ be continuous, then we can define
our dynamical system as
\begin{equation}
\dot{x}=f(x)\text{,}\qquad f(x+1)=f(x)\;\forall\,x\in\mathbb{R}.
\label{Eq: original}%
\end{equation}

\begin{figure}[htb]
\centering
\includegraphics[scale=.65]{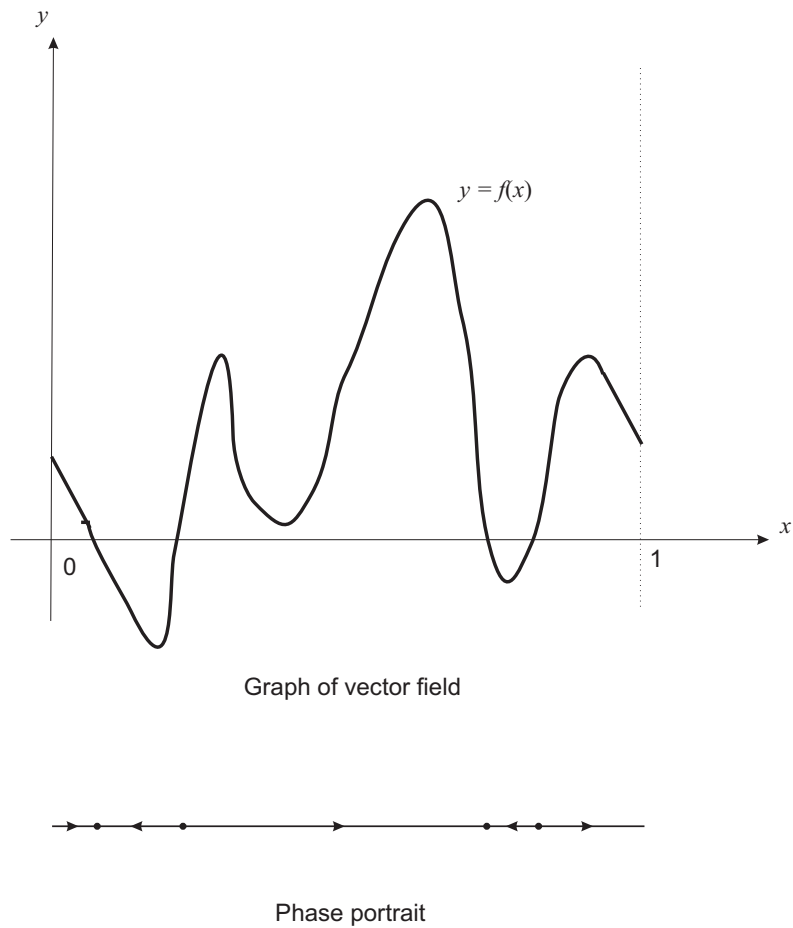}\caption{Example of a vector field of a
structurally stable dynamical system on the circle $\mathbb{S}$, represented
as the unit interval, and its phase portrait.}
\label{Fig: orig}
\end{figure}

That is, $f$ is a periodic function of period one. This can be
simplified by restricting to one period, namely the unit interval $0\leq
x\leq1$ such that $f(0)=f(1)$. Then (\ref{Eq: original}) becomes

\begin{equation}
\dot{x}=f(x)\text{,}\qquad f(1)=f(0). \label{Eq: f}%
\end{equation}

In this context the function $f$ is called a \emph{vector field}. The vector
field on $\mathbb{S}^{1}$ of class $C^{k}$ is a function that has $k$
continuous derivatives where each derivative is identified at the end points;
i.e. $f:[0,1]\rightarrow\mathbb{R}$ such that $f^{(m)}(0)=f^{(m)}(1)$
$\forall\,0\leq m\leq k$. An example of this would be the graph in Fig
\ref{Fig: orig}.

\section{Some Key Definitions}

Before we begin discussing our main results, let us introduce some definitions
that are to play key roles in our work. We shall assume in the sequel that all
of our dynamical systems are at least $C^{1}$.

\begin{definition}
A fixed point $x_{\ast}$ of (\ref{Eq: f}) is one such that $f(x_{\ast})=0$,
and is said to be \emph{hyperbolic} if $f^{\prime}(x_{\ast})\neq0$, otherwise
it is said to be \emph{nonhyperbolic}.
\end{definition}

\begin{definition}
A map $h:M\rightarrow M$, where $M$ is a manifold, is said to be a
\emph{homeomorphism} if it is a bijective bicontinuous map.
\end{definition}

For example in \cite{Meiss07} it is shown the maps
$h:(0,\infty)\rightarrow(0,1)$ defined by $h=1/(1+x^{2})$
and $h:\mathbb{S}^{1}\rightarrow \mathbb{S}^{1}$
defined by $h(x)=x+a\cos x$ for $|a|<1$ are
\emph{homeomorphisms}.\\

\begin{definition}
Two dynamical systems $\dot{x}=f(x)$ and $\dot{y}=g(y)$ on $M$ are
\emph{topologically equivalent} if there is a homeomorphism $h$ of $M$ such
that $h $ maps oriented (by incresing time) orbits of the first system onto
oriented orbits of the second system. Such an $h$ is called a
\emph{topological equivalence} between the systems (or sometimes a
\emph{conjugacy} - especially for the case of discrete dynamical systems).
\end{definition}

It is easy to verify that a rotation of $-\pi/2$ radians, which corresponds to
$y=h(x)=x-1/4$ for the $\operatorname{mod}1$ representation, is a topological
equivalence between the dynamical systems $\dot{\theta}=\sin\theta$ and
$\dot{\phi}=\cos\phi$ on the circle $\mathbb{S}^{1}$. A two-dimensional
example (in the plane $\mathbb{R}^{2}$) in \cite{Perko01} shows that the
linear system $\dot{x}=Ax$ is \emph{topologically equivalent} to $\dot{y}=By$
where
\[
A=\left[
\begin{array}
[c]{cc}%
-1 & -3\\
-3 & -1
\end{array}
\right]  \quad\text{and}\quad B=\left[
\begin{array}
[c]{cc}%
2 & 0\\
0 & -4
\end{array}
\right]
\]
via the homeomorphism
\[
h(x)=\frac{1}{\sqrt{2}}\left[
\begin{array}
[c]{cc}%
1 & -1\\
1 & 1
\end{array}
\right]  x.
\]
To see this we note that the origin is the only fixed point of both systems,
and for any other $(x_{1}^{0},x_{2}^{0})$, $h$ maps the solution of $\dot
{x}=Ax$ beginning at this initial point onto the solution $(y_{1}^{0}%
e^{2t},y_{2}^{0}e^{-4t})$, with $(y_{1}^{0},y_{2}^{0})=h\left(  (x_{1}%
^{0},x_{2}^{0})\right)  =1/\sqrt{2}\left(  x_{1}^{0}-x_{2}^{0},x_{1}^{0}%
+x_{2}^{0}\right)  .$\newline

\begin{definition}
Denote the class of $C^{1}$ maps of the circle by $C^{1}(\mathbb{S}^{1})$. Then

\[
\left\Vert f\right\Vert _{1}:=\sup\{\left\vert f(x)\right\vert :x\in
\mathbb{S}^{1}\}+\sup\{\left\vert f^{\prime}(x)\right\vert :x\in\mathbb{S}%
^{1}\}
\]

defines a norm on $C^{1}(\mathbb{S}^{1})$ called the \emph{$C^{1}$ norm}. This
norm generates a topology, called the \emph{$C^{1}$-topology}, in the usual
way via the open $\epsilon$-balls $B_{\epsilon}^{1}(f):=\{g\in C^{1}%
(\mathbb{S}^{1}):\left\Vert g-f\right\Vert _{1}<\epsilon\}$.
\end{definition}

\begin{definition}
The dynamical systems $\dot{x}=f(x)$ with $f\in C^{1}(\mathbb{S}^{1})$ is said
to be $C^{1}$ \emph{structurally stable} if for every $\epsilon>0$
sufficiently small, all $g\in B_{\epsilon}^{1}(f)$ are topologically
equivalent to $f$.
\end{definition}

We note that it is not difficult to imagine how these last two definitions can
be generalized to any finite-dimensional closed (compact and without boundary) manifolds.

It is useful to take note of the following rather simple characterization
of topological equivalence for a pair of $C^{1}$ dynamical systems, (i)
$\dot{x}=f(x)$ and (ii) $\dot{y}=g(y)$ on the unit circle $\mathbb{S}^{1}$:
\textit{A homeomorphism }$h:S^{1}\rightarrow S^{1}$\textit{\ is a topological
equivalence from }(i)\textit{\ to }(ii) \textit{if and only if}%
\begin{align}
h\left(  \{x\in\mathbb{S}^{1}:f(x)=0\}\right)   &  =\{y\in\mathbb{S}%
^{1}:g(y)=0\},\;h\left(  \{x\in\mathbb{S}^{1}:f(x)>0\}\right)  =\{y\in
\mathbb{S}^{1}:g(y)>0\}\text{ and}\nonumber\\
\text{ }h\left(  \{x\in\mathbb{S}^{1}:f(x)<0\}\right)   &  =\{y\in
\mathbb{S}^{1}:g(y)<0\}. \tag{TE}\label{TE}%
\end{align}

\subsection{Bump functions}

\label{Sec: bump} Throughout the necessity portion of the proof of the
one-dimensional analog of Theorem P, we will be using bump functions. A simple
example of a bump function is $\psi:\mathbb{R}\rightarrow\mathbb{R}$ defined
as (\ref{Eq: simple bump}) and plotted in Fig \ref{Fig: bump}.

\begin{equation}
\psi(x):=%
\begin{cases}
e^{\frac{-1}{1-x^{2}}} & \text{for}\;x\in(-1,1),\\
0 & \text{for}\;x\notin(-1,1);
\end{cases}
\label{Eq: simple bump}%
\end{equation}

\begin{figure}[htb]
\centering
\includegraphics[scale=.5]{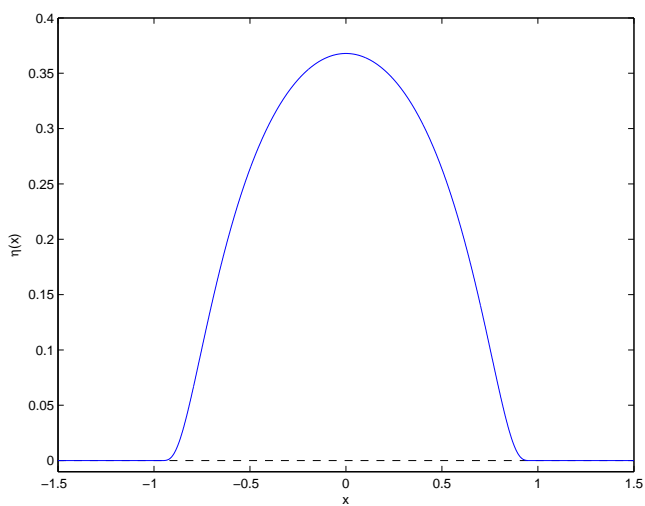}\caption{Example of a bump function.}%
\label{Fig: bump}%
\end{figure}

\indent Now, we can restrict the bump to any ball, which for our
proof will correspond to an interval of our choice. Let $x_{0}$ be the center
of the ball, and $r$ $(>0)$ be the radius. Then, denoting the ball as
$B_{r}\left(  x_{0}\right)  :=\{x\in\mathbb{R}:$ $\left\vert x-x_{0}%
\right\vert <r\}$, our bump function becomes

\begin{equation}
\psi(x):=%
\begin{cases}
\exp\left(  \frac{-r^{2}}{r^{2}-(x-x_{0})^{2}}\right)  & \text{for}\;x\in
B_{r}(x_{0}),\\
0 & \text{for}\;x\notin B_{r}(x_{0});
\end{cases}
\end{equation}

\indent This can be specified to any interval $[a,b]$, which corresponds to
a ball centered at $(a+b)/2$ with a radius of $(b-a)/2$: namely,

\begin{equation}
\psi(x):=%
\begin{cases}
\exp\left(  -\left(  \frac{b-a}{2}\right)  ^{2}\middle/\left[  \left(
\frac{b-a}{2}\right)  ^{2}-\left(  x-\frac{a+b}{2}\right)  ^{2}\right]
\right)  & \text{for}\;x\in(a,b),\\
0 & \text{for}\;x\notin(a,b);
\end{cases}
\end{equation}

\indent We also note that that it can be easily verified using basic calculus
that the bump function (6) satisfies

\begin{equation}
\left\Vert \psi\right\Vert _{1}<e^{-1}\left[  1+6e^{-1}(b-a)^{-1}\right]  ,
\label{Eq: bounded}%
\end{equation}

so once the interval has been specified, for every $\epsilon>0$ there exists
a $\sigma>0$ such that $\left\Vert \sigma\psi\right\Vert _{1}<\epsilon$. This
will be an important fact to keep in mind for the necessity portion of our
proof. In many instances we will use a scaled bump function to obtain small
$C^{1}$ perturbations of given functions in $C^{1}(\mathbb{S}^{1})$.

\section{Peixoto's Theorem on $\mathbb{S}^{1}$}

We note that every closed one-dimensional manifold is homeomorphic to a circle,
as shown in \cite{Milnor97}. So, it suffices to restrict our attention to $\mathbb{S}^{1}$
in the one-dimensional analog of Theorem P that follows.

\bigskip

\begin{thm}
\label{Thm: peixoto} Suppose (\ref{Eq: f}) is a $C^{1}$ dynamical system on
$\mathbb{S}^{1}$. Then (\ref{Eq: f}) is $C^{1}$ structurally stable if and
only if it has finitely many fixed points, all of which are hyperbolic.
\end{thm}

\begin{proof}
For sufficiency, let us first prove the result for a dynamical system with no
fixed points. Suppose (\ref{Eq: f}) has no fixed points, then $|f(x)|>0$ on
$[0,1]$. With out loss of generality, assume $f$ is positive, which means the
phase space consists of a single periodic counterclockwise orbit. Since $f$ is
continuous, there is an $\epsilon_{0}>0$ such that $f(x)>\epsilon_{0}%
\;\forall\,x\in\lbrack0,1]$. Consider the dynamical system
\begin{equation}
\dot{y}=g(y), \label{Eq: g}%
\end{equation}
where $g\in C^{1}(\mathbb{S}^{1})$. If
\begin{equation}
\left\Vert g-f\right\Vert _{1}<\epsilon<\epsilon_{0}, \label{Eq: C1}%
\end{equation}
then $g$ must also be positive on $\mathbb{S}^{1}$. Therefore, it follows from
(\ref{TE}) that the identity map is a topological equivalence, so
(\ref{Eq: f}) is structurally stable.

Next we prove the result for finitely
many hyperbolic fixed points. Suppose $f(x_{k})=0$ and $f^{\prime}(x_{k}%
)\neq0$ for $k=1,2,\ldots,m$. We may assume none of them is an endpoint of the
unit interval, since this can always be accomplished via a simple translation
of the period interval. Let us order them as follows:
\[
x_{k}<x_{k+1},k=1,2,\ldots,m-1\quad\text{and}\quad x_{m}<x_{1}.
\]
We use this ordering to make it more understandable when speaking about
\textquotedblleft sequential\textquotedblright\ points. Notice, the end points
are between the fixed points $x_{m}$ and $x_{1}$ and the non-fixed orbits are
just open intervals between sequential points.

We observe that between
any two sequential fixed points, $f$ does not change sign, and in some $\delta
$-neighborhood of every fixed point, $f^{\prime}$ is nonzero owing to the
hyperbolicity and continuity. So, we can select $\epsilon_{0}$ and $\delta>0$ small
enough such that $|f^{\prime}(x)|\geq2\epsilon_{0}$ on $[x_{k}-\delta
,x_{k}+\delta]$ for $k=1,2,\ldots,m$, where the intervals $[x_{k}-\delta
,x_{k}+\delta]$ are disjoint. Hence, $f$ is monotonic and has a single zero on
each of these intervals. Now define $K(\delta)$ as the closure of the
complement of these intervals, which is just a disjoint union of closed
intervals itself. Since $f$ does not change sign between any two sequential
fixed points, $f$ is nonvanishing in $K(\delta)$. Furthermore, it follows from
continuity of $f$ that there is $0<\epsilon_{1}<\epsilon$ such that
$|f(x)|\geq2\epsilon_{1}\;\forall\,x\in K(\delta)$.

Now we show that
structural stability is satisfied; in particular, for any $g\in C^{1}%
(\mathbb{S}^{1})$ satisfying
\begin{equation}
\left\Vert f-g\right\Vert _{1}<\epsilon<\min\{\epsilon_{0},\epsilon_{1}\}
\label{Eq: pert}%
\end{equation}
the dynamical system (\ref{Eq: g}) is topologically equivalent to
(\ref{Eq: f}). That is, there is a homeomorphism $h:\mathbb{S}^{1}%
\rightarrow\mathbb{S}^{1}$ mapping oriented orbits of (\ref{Eq: g}) to
(\ref{Eq: f}). By (\ref{Eq: pert}), $g$ has exactly one zero, denoted as
$y_{k}$, on each interval $[x_{k}-\delta,x_{k}+\delta]$ and no zeros on
$K(\delta)$, and furthermore $g^{\prime}(y_{k})\neq0$ and $g$ has the same
sign on $(y_{k},y_{k+1})$ as $f$ does on $(x_{k},x_{k+1})$ whenever $1\leq
k\leq m$. We select $h$ to be a piecewise linear homeomorphism with the
following properties: $h(0)=h(1)=0$ and also $h(x_{k})=y_{k}$ for
$k=1,2,\ldots,m$, and linear on the intervals $(0,x_{1}),(x_{1},x_{2}%
),\ldots,(x_{m-1},x_{m}),(x_{m},1)$; namely%
\[
h(x):=\left\{
\begin{array}
[c]{cc}%
\left(  \frac{y_{1}}{x_{1}}\right)  x, & 0\leq x\leq x_{1}\\
\frac{\left(  y_{k+1}-y_{k}\right)  x+\left(  y_{k}x_{k+1}-x_{k}%
y_{k+1}\right)  }{\left(  x_{k+1}-x_{k}\right)  }, & x_{k}\leq x\leq
x_{k+1},1\leq k\leq m-1\\
\left(  \frac{y_{m}}{x_{m}-1}\right)  (x-1), & x_{m}\leq x\leq1
\end{array}
\right.
\]
This proves the sufficiency of the hypothesis owing to (\ref{TE}).\\ \\

For necessity, we show that if the fixed point hypothesis is not satisfied, then
(\ref{Eq: f}) is not structurally stable. In order to show this, we need to
analyze all the cases in which the hypothesis may be violated. In each case we
show if we add an arbitrarily small $C^{1}$ perturbation, $\eta$, of the right
form to the original system, $g(x)=f(x)+\eta(x)$, we obtain a system that is
not topologically to (\ref{Eq: f}). The analysis is to be done in
neighborhoods of nonhyperbolic fixed points, with the intention of showing
that an arbitrarily small perturbation can change the homeomorphism type of
the original fixed point set, thus insuring in view of (\ref{TE}) that the
perturbed system cannot be topologically equivalent to the original.\\

\noindent Case 1: An isolated nonhyperbolic fixed point $x_{\ast}$ across which $f$ does
not change sign. It can be assumed, without loss of generality, that $f(x)>0$
in some punctured neighborhood $\mathring{B}_{\delta}(x_{\ast}):=B_{\delta
}(x_{\ast})\smallsetminus\{x_{\ast}\}$ of $x_{\ast}$ as shown in Fig
\ref{Fig: case1}.
\begin{figure}[htb]
\centering
\includegraphics[scale=.5]{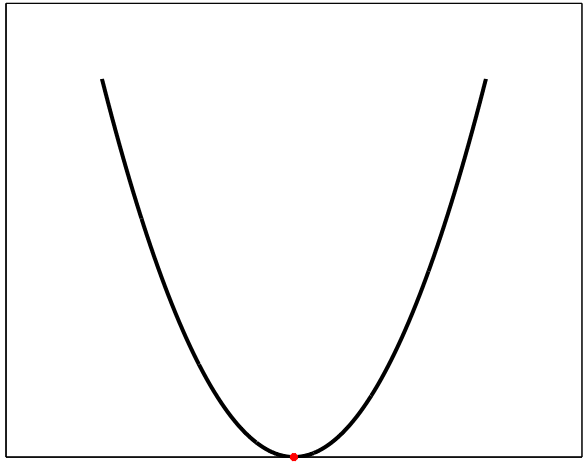}\caption{An example of a function in a
neighborhood of a nonhyperbolic fixed point with the aforementioned
properties.}
\label{Fig: case1}
\end{figure}
Note that if $x_{\ast}$ is the only fixed point, the perturbation
$g(x)=f(x)+\epsilon$ defines a dynamical system with an empty fixed point set,
denoted as $g^{-1}(0)=\varnothing$, while $f^{-1}(0)=\{x_{\ast}\}$, so it
follows from (TE) that $\dot{x}=f$ and $\dot{y}=g$ are topologically inequivalent.

If $x_{\ast}$ is not unique, we have to localize the above analysis.
In order to do this we would like to perturb the
system in such a way as to annihilate the fixed point without affecting
anything outside of the neighborhood. This can be accomplished by using a bump
function (cf. section \ref{Sec: bump}).

Since $x_{\ast}$ is the
only fixed point, which we may assume is not an end point of the unit
interval, in $(x_{\ast}-\delta,x_{\ast}+\delta)$, we define

\begin{equation}
\psi(x):=
\begin{cases}
\exp\left(  \frac{-\delta^{2}}{\delta^{2}-(x-x_{\ast})^{2}}\right)  &
\text{for}\;x\in(x_{\ast}-\delta,x_{\ast}+\delta),\\
0 & \text{for}\;x\notin(x_{\ast}-\delta,x_{\ast}+\delta);
\end{cases}
\label{Eq: case1}
\end{equation}

It follows from (\ref{Eq: bounded}) that for any $\epsilon>0$, there exists a
$\sigma>0$ such that $\left\Vert \sigma\psi\right\Vert _{1}<\epsilon$. Hence,
if $g=f+\sigma\psi$, $\left\Vert f-g\right\Vert _{1}<\epsilon$ and $\dot{y}=g$
has no fixed point in $(x_{\ast}-\delta,x_{\ast}+\delta)$. Consequently,
$f^{-1}(0)$ and $g^{-1}(0)$ cannot be homeomorphic, which means that $f$ is
not structurally stable.\\

\noindent Case 2: An isolated nonhyperbolic fixed point $x_{\ast}$ across
which $f$ changes sign, which can be assumed, without loss of generality,
to be as shown in Fig. \ref{Fig: case2}. Again, we define $\delta>0$ such
that $f$ is not zero in $\mathring{B}_{\delta}(x_{\ast})$ .
\begin{figure}[htb]
\centering
\includegraphics[scale=.5]{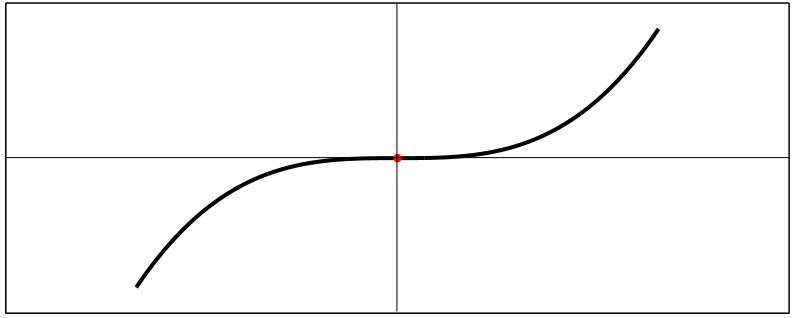}
\caption{An example of a function in a
neighborhood of a nonhyperbolic fixed point with the aforementioned
properties.}%
\label{Fig: case2}%
\end{figure}
For this case, we shall show there is an arbitrarily small
$C^{1}$ perturbation confined to $B_{\delta}(x_{\ast})$, which has three hyperbolic
fixed points in this interval instead of one nonhyperbolic one.

\begin{figure}[htb]
\centering
\includegraphics[scale=.4]{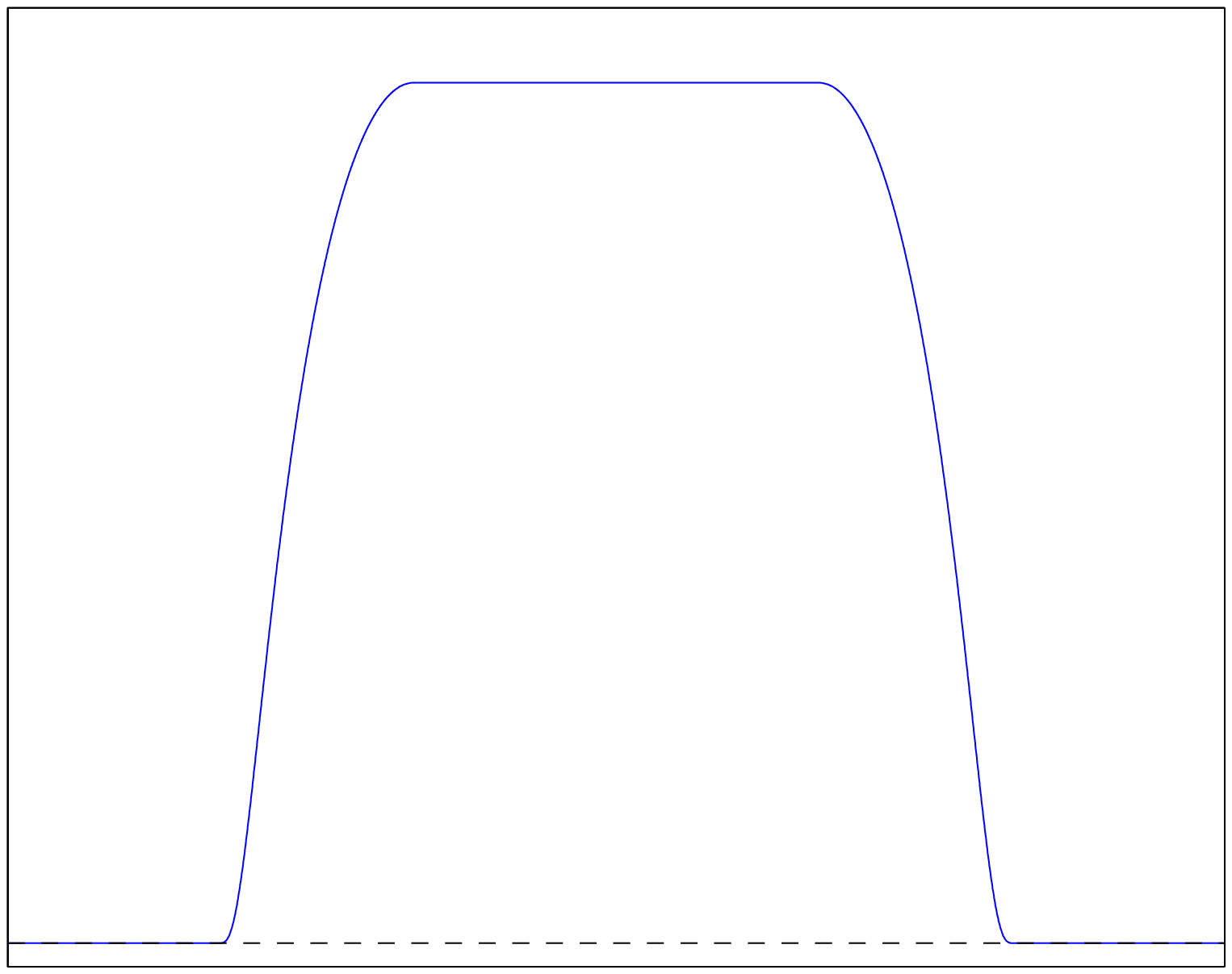}
\includegraphics[scale=.4]{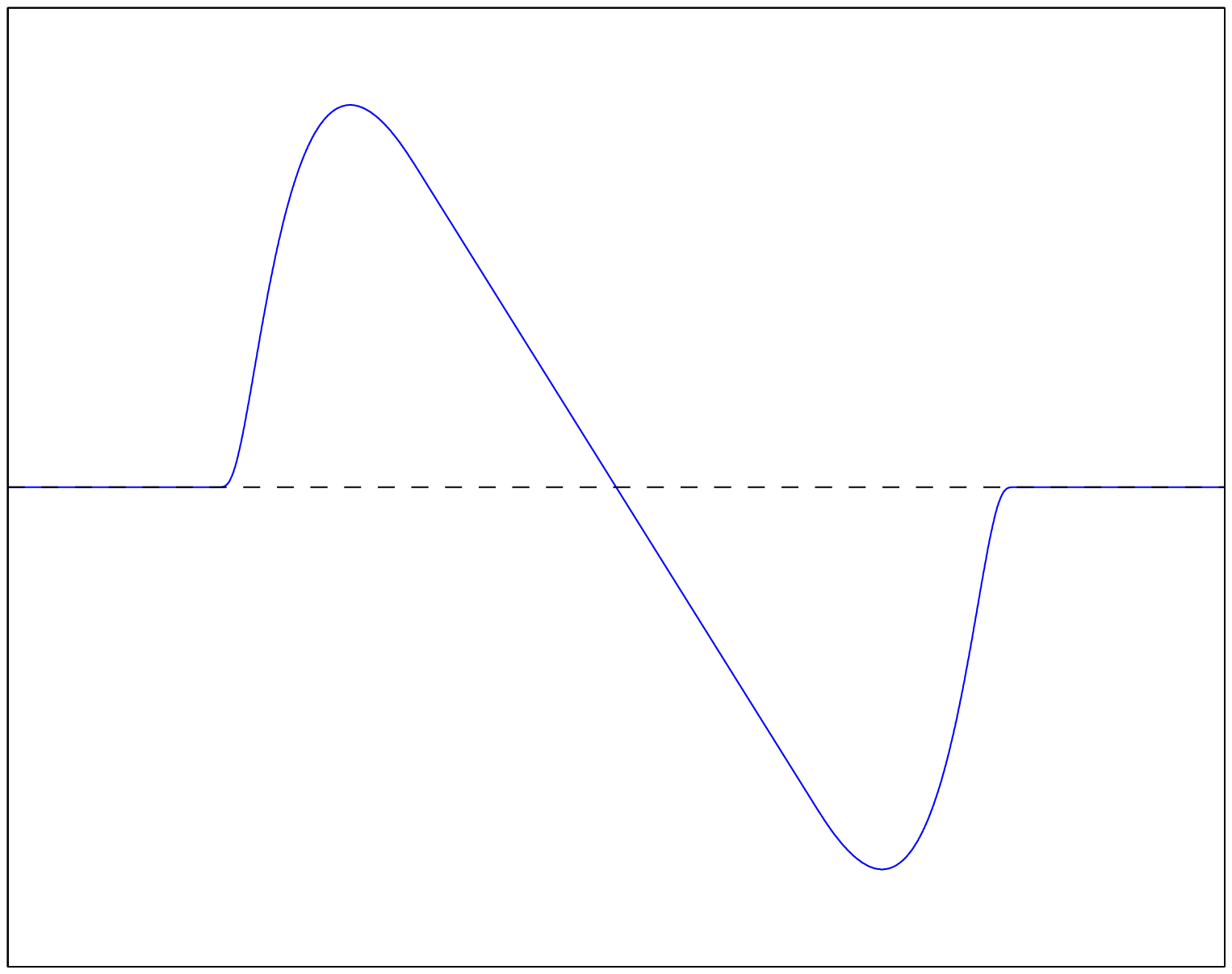}
\caption{Example plots of equations \ref{Eq: case2} and \ref{Eq: case2v} respectively.}
\label{Fig: case2a}
\end{figure}

First, we put two bump functions together to create a new bump
function that is equal to $e^{-1}$ in the closed interval $[x_{\ast}%
-\delta/2,x_{\ast}+\delta/2]$ and vanishes in the complement of $(x_{\ast
}-\delta,x_{\ast}+\delta)$; namely,
\begin{equation}
\varphi(x):=\left\lbrace
\begin{array}
[c]{cc}%
\exp\left(  -\left(  \frac{\delta}{2}\right)  ^{2}\middle/\left[  \left(
\frac{\delta}{2}\right)  ^{2}-\left(  x-x_{\ast}+\frac{\delta}{2}\right)
^{2}\right]  \right)  , & x_{\ast}-\delta<x\leq x_{\ast}-\delta/2\\
e^{-1}, & x_{\ast}-\delta/2\leq x\leq x_{\ast}+\delta/2\\
\exp\left(  -\left(  \frac{\delta}{2}\right)  ^{2}\middle/\left[  \left(
\frac{\delta}{2}\right)  ^{2}-\left(  x_{\ast}-x+\frac{\delta}{2}\right)
^{2}\right]  \right)  , & x_{\ast}+\delta/2\leq x<x_{\ast}+\delta\\
0, & x\notin(x_{\ast}-\delta,x_{\ast}+\delta)
\end{array}\right\rbrace .  
\label{Eq: case2}
\end{equation}
Observe that this function is $C^{\infty}$ on the whole real line except at
the points $x_{\ast}-\delta/2$ and $x_{\ast}+\delta/2$ where it is only
$C^{1}$. Next, we define
\begin{equation}
\vartheta(x) := -\frac{2(  x-x_{\ast})\varphi(x)}{\delta e}
\label{Eq: case2v}
\end{equation}
Examples of (\ref{Eq: case2}) and (\ref{Eq: case2v}) are given in Fig \ref{Fig: case2a}.

Note, as usual, for each $\epsilon>0$ there exists a $\sigma>0$ such that
$\left\Vert \sigma\vartheta\right\Vert _{1}<\epsilon$. 
Moreover,
$g:=f+\sigma\vartheta$ has a zero at $x_{\ast}$ with $g^{\prime}(x_{\ast})<0$
and just two other zeros in $B_{\delta}(x_{\ast})$ at points $x_{\ast}\pm\nu$,
with $0<\nu<\delta$, and $g^{\prime}(x_{\ast}-\nu)=g^{\prime}(x_{\ast}+\nu
)>0$. As $\left\Vert f-g\right\Vert _{1}<\epsilon$ and $f^{-1}(0)$ and
$g^{-1}(0)$ are not homeomorphic, (TE) implies that $f$ is not structurally stable.\\

\noindent Case 3: An interval of (nonhyperbolic) fixed points, i.e. $f(x)=0$
on some interval $[a,b]\subseteq\lbrack0,1]$. If $[a,b]=[0,1]$, the addition
of an arbitrarily small positive constant changes the fixed point set from all
of $\mathbb{S}^{1}$ to the empty set, which proves that such a system cannot
be $C^{1}$structurally stable. On the other hand, if the interval is a proper
subset of the unit interval, we may assume that $[a,b]\subseteq\left(
0,1\right)  $ and it is isolated from any other points in the fixed point set
of $\dot{x}=f$. Accordingly, there is a positive $\delta$ such that
$[a-\delta,b+\delta]\subset\left(  0,1\right)  $ and $[a-\delta,b+\delta]\cap
f^{-1}\left(  0\right)  =[a,b]$. By analogy with Case 1 and Case 2 above, we
consider two subcases: (i) $f$ has the same sign in $(a-\delta,a)$ and
$(b,b+\delta)$; and (ii) $f$ has opposite signs in $(a-\delta,a)$ and
$(b,b+\delta)$. Naturally, we may assume without loss of generality the the
sign in (i) is positive, and in (ii) it goes from negative to positive. It is
convenient to use the following analog of the bump function $\varphi$ for both
(i) and (ii):

\begin{equation}
\hat{\varphi}(x):=\left\{
\begin{array}
[c]{cc}%
\exp\left(  -\left(  \frac{\delta}{2}\right)  ^{2}\middle/\left[  \left(
\frac{\delta}{2}\right)  ^{2}-\left(  x-a+\frac{\delta}{2}\right)
^{2}\right]  \right)  , & a-\delta<x\leq a-\delta/2\\
e^{-1}, & a-\delta/2\leq x\leq b+\delta/2\\
\exp\left(  -\left(  \frac{\delta}{2}\right)  ^{2}\middle/\left[  \left(
\frac{\delta}{2}\right)  ^{2}-\left(  b-x+\frac{\delta}{2}\right)
^{2}\right]  \right)  , & b+\delta/2\leq x<b+\delta\\
0, & x\notin(a-\delta,b+\delta)
\end{array}
\right.  . \label{Eq: case3}%
\end{equation}

Inasmuch as for any $\epsilon>0$ there is a $\sigma>0$ such that $\left\Vert
\sigma\phi\right\Vert _{1}<\epsilon$, in subcase (i) the perturbation $\dot
{y}=g=f+\sigma\phi$ satisfies has $\left\Vert f-g\right\Vert _{1}<\epsilon$
and has no fixed points in $[a-\delta,b+\delta]$, which means it cannot be
topologically equivalent to $\dot{x}=f$ in virtue of (TE). While for subcase
(ii), the perturbation $\dot{y}=g$, where
\[
g:=f+\sigma\hat{\vartheta},
\]
with
\begin{equation}
\hat{\vartheta}(x):=-\frac{2\hat{\varphi}(x)}{e\left(  a+b\right)  }\left(
x-\frac{a+b}{2}\right)  ,
\label{Eq: case3v}
\end{equation}
produces an arbitrarily small $C^{1}$ perturbation of $\dot{x}=f$ having
precisely three hyperbolic fixed points in $[a-\delta,b+\delta]$. Therefore,
$f$ is not structurally stable for any of these subcases.\\ \\

Although there remain some situations for which the hypothesis of the theorem
does not hold, we have actually introduced the bump function methodology which
can be readily seen to be capable of disposing of the remaining cases.
Consequently, we leave the remaining details of the proof to the reader, with
the following very interesting exception.\\ \\

\noindent Case 4: Suppose $\dot{x}=f$ has distinct fixed points $x_{1}%
,x_{2},x_{3},\ldots$ and $x_{n}\rightarrow x_{\ast}$ , so that $x_{\ast} $ is
a fixed point and the limit of the sequence $\{x_{n}\}$, which may be assumed
to lie in an open interval $J=(x_{\ast}-r,x_{\ast}+r)$ contained in $(0,1)$,
which contains no other fixed points. The sequence might consist of all
hyperbolic fixed points as shown in Fig \ref{Fig: countable}, or it might be
comprised of some combination of hyperbolic fixed points and and nonhyperbolic
fixed points of the types treated in Cases 1 and 2.

\begin{figure}[htb]
\centering
\includegraphics[scale=.5]{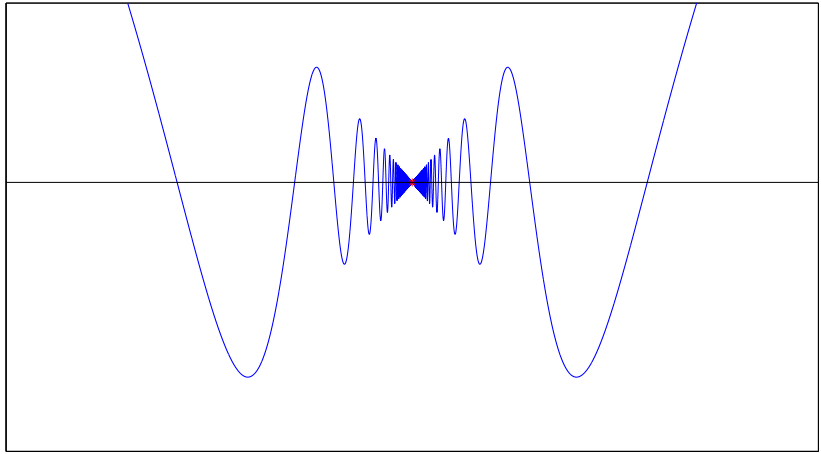}\caption{An example of a function
with fixed points having the aforementioned properties.}
\label{Fig: countable}
\end{figure}

Notice that $\{x_{n}\}$ is countably infinite, so that if we perturb the
system $\dot{x}=f$ to a system $\dot{y}=g$ with only finitely many fixed
points in $J$, and the same fixed points in the complement of $J$, the two
systems must be topologically inequivalent.

Since $f$ is $C^{1}$, for any $\epsilon>0$ there is a positive
$\delta=\delta(\epsilon)<\epsilon$ such that $|x-x_{\ast}|<\delta$ implies
$\left\vert f(x)\right\vert =|f(x)-f(x_{\ast})|<\epsilon\left\vert x-x_{\ast
}\right\vert $. Furthermore, there are only finitely many fixed points in
$J\smallsetminus B_{s}(x_{\ast})$ for any $0<s<r$. Let us use the bump
function
\begin{equation}
\psi(x):=%
\begin{cases}
\exp\left(  \frac{-\delta^{2}}{\delta^{2}-(x-x_{\ast})^{2}}\right)  &
\text{for}\;x\in(x_{\ast}-r/2,x_{\ast}+r/2),\\
0 & \text{for}\;x\notin(x_{\ast}-r/2,x_{\ast}+r/2);
\end{cases}
,
\end{equation}
and for any given $\epsilon>0$ choose $\sigma>0$ such that $\left\Vert
g-f\right\Vert _{1}=\left\Vert \left(  f+\sigma\upsilon\right)  -f\right\Vert
_{1}<\epsilon$. Then $g$ has no zeros in $B_{s}(x_{\ast})$ for some $0<s<r/2$
and so only finitely many fixed points in $J$, which means that $f$ is not
structurally stable. This completes our proof.

\end{proof}

\section{Density Theorem on $\mathbb{S}^{1}$}

We now prove the one-dimensional analog of the density part of Theorem P. It
is convenient to introduce the following notation towards this end. Define
$SS^{1}\left(  \mathbb{S}^{1}\right)  $ to be the $C^{1}$-structurally stable
systems $\dot{x}=f$ on $C^{1}\left(  \mathbb{S}^{1}\right)  .$

\begin{thm}
The set of dynamical systems $SS^{1}\left(  \mathbb{S}^{1}\right)  $ is
$C^{1}$ open and dense in $C^{1}\left(  \mathbb{S}^{1}\right)  .$
\end{thm}

\begin{proof}

The openness follows directly from the sufficiency proof of Theorem 1, and the
density is essentially a straightway consequence of the necessity argument for
the same theorem. In particular, it was shown in the sufficiency proof that
the fixed point hypothesis is preserved under sufficiently small
perturbations, and so $SS^{1}\left(  \mathbb{S}^{1}\right)  $ is a $C^{1}$
open subset of $C^{1}\left(  \mathbb{S}^{1}\right)  $.

It is clear from the methods used in proving the necessity part of Theorem 1,
that for any $C^{1}$ dynamical system $\dot{x}=f$ on $\mathbb{S}^{1}$ there is
an arbitrarily small $C^{1}$ perturbation $\dot{y}=g$ such that $g$ has only
finitely many zeros. Then, using the bump function methods employed for Cases
1 and 2 of the necessity in Theorem 1, we can obtain a further arbitrarily
small perturbation $\dot{z}=h$ with only hyperbolic fixed points. This
completes the proof.

\end{proof}

It is worth noting that one could have used several other types of bump
function based perturbations in the above proofs of the necessity of the
hyperbolic hypothesis in Theorem 1 and the density result in Theorem 2. For
example, the functions $\vartheta$ and $\hat{\vartheta}$ used for Case2 and
Case 3 (ii), respectively, in the necessity proof of Theorem 1 could be
replaced with an appropriate form of the derivative of a bump function, as is
evident from the plot of the first and second derivatives of the simple bump
function (\ref{Eq: simple bump}) given in Fig \ref{Fig: dbump}.

\begin{figure}[htb]
\centering
\includegraphics[scale=.4]{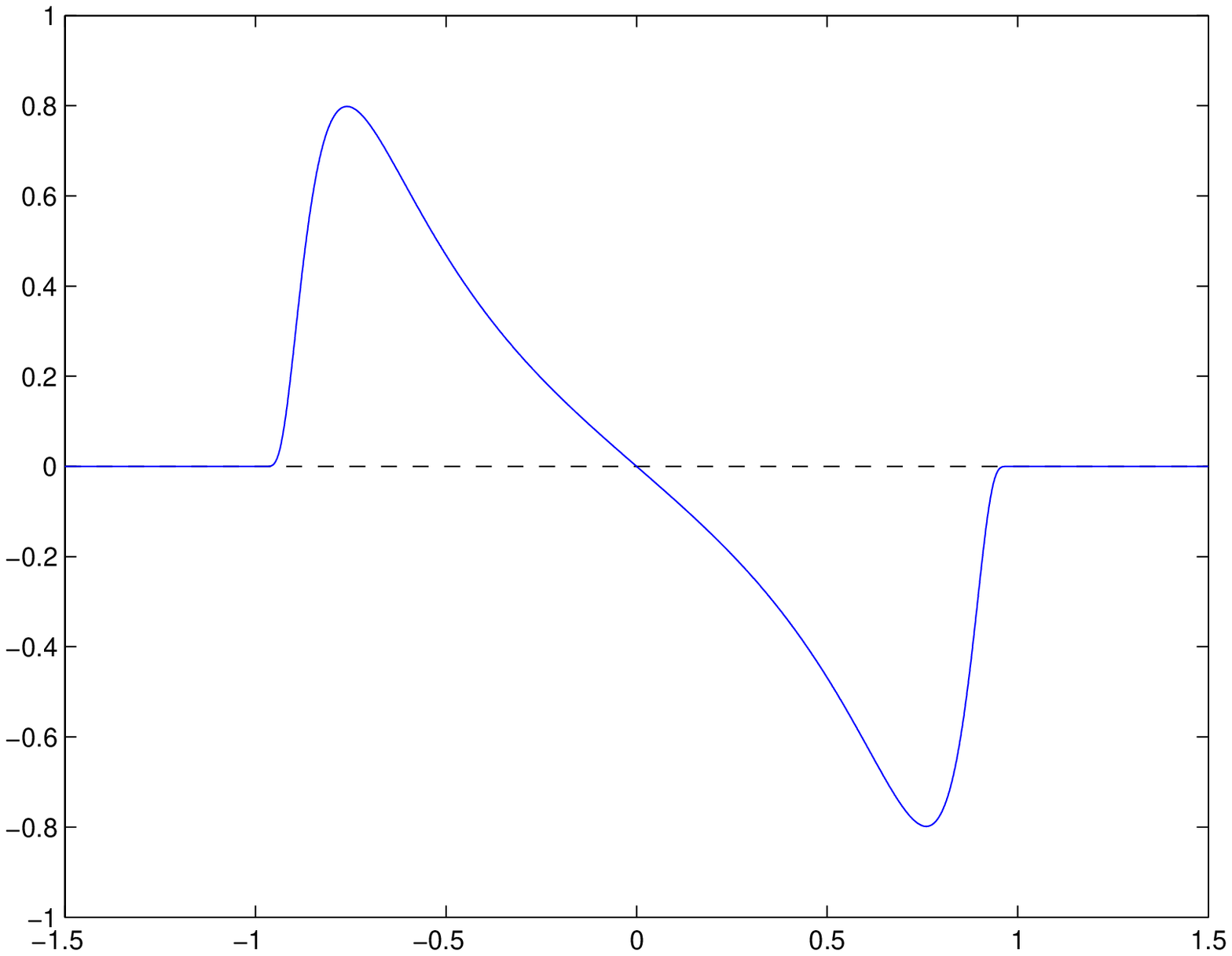}
\includegraphics[scale=.4]{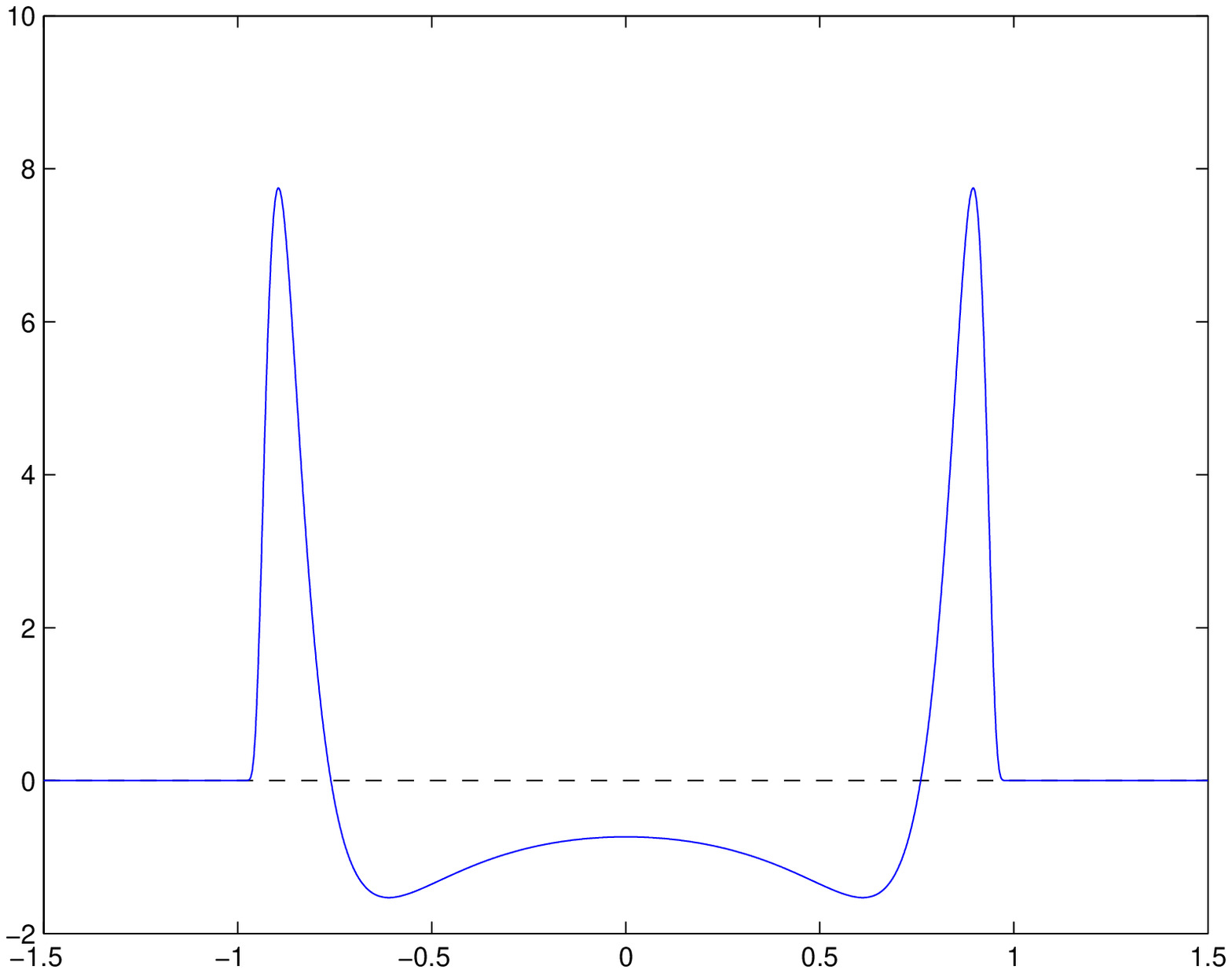}\caption{The first and second
derivatives, respectively, of the simple bump function.}
\label{Fig: dbump}
\end{figure}

Finally, it is interesting to remark that the proofs of both Theorems 1 and 2
can be reduced to just a few lines by the application of a standard
transversality theorem described for example in \cite{Guillemin-Pollack74}, which
is an indication of the importance of differential topology in the modern
theory of dynamical systems.

\section{Conclusion}

The epochal structural stability and density theorems of Peixoto for dynamical
systems on closed surfaces have long and complicated proofs involving concepts
unfamiliar to many undergraduate enthusiasts. In this note we have demonstrated
that the one-dimensional analogs of these theorems can be proved using methods
that are well known to most undergraduate mathematics majors, thus providing a
useful introduction to many of the elements of the two-dimensional proofs. One
might well imagine that Peixoto himself considered the one-dimensional version
and used it, along with the pioneering efforts of Andronov and Pontryagin, as a 
guide for his theorems.


\section{Acknowledgement}

I would like to thank Professor Denis Blackmore of the Department of Mathematical
Sciences at New Jersey Institute of Technology for suggesting this project -
which he often sketched when teaching dynamical systems courses - and for his
constant encouragement and advice in bringing it to fruition. In addition, I would like 
to thank my friends Tom, Heather, and Ivana for giving me early feedback on the
readability of this note.


\pagebreak

\nocite{Strogatz94}
\bibliographystyle{plain}
\bibliography{DynSys}

\end{document}